\documentclass[10pt]{amsart}

\usepackage{amsmath}
  \usepackage{paralist}
  \usepackage{graphics} 
  \usepackage{epsfig} 
\usepackage{graphicx}  \usepackage{epstopdf}

 \usepackage[colorlinks=true]{hyperref}
\hypersetup{urlcolor=blue, citecolor=red}

\newtheorem{theorem}{Theorem}[section]
\newtheorem{corollary}[theorem]{Corollary}
\newtheorem{main}{Main Theorem}
\newtheorem*{main*}{Main Theorem}
\newtheorem{lemma}[theorem]{Lemma}
\newtheorem{proposition}[theorem]{Proposition}

\theoremstyle{definition}
\newtheorem{definition}[theorem]{Definition}
\newtheorem{remark}[theorem]{Remark}

\title[Running heading with forty characters or less]
      {Higher rank rigidity for Berwald spaces}

\author[]{Weisheng Wu}


\address{Department of Applied Mathematics, College of Science, China Agricultural University, Beijing, 100083, P.R. China}
 \email{wuweisheng@cau.edu.cn}

\begin{document}
\maketitle
\markboth{Higher rank rigidity for Berwald spaces}
{Weisheng Wu}
\renewcommand{\sectionmark}[1]{}

\begin{abstract}
We generalize the higher rank rigidity theorem to a class of Finsler spaces, i.e. Berwald spaces. More precisely, we prove that a complete connected Berwald space of finite volume and bounded nonpositive flag curvature with rank at least $2$ whose universal cover is irreducible, is a locally symmetric space or a locally Minkowski space.
\end{abstract}

\section{Introduction}
In 1980's, based on the work in Ballmann-Brin-Eberlein \cite{BBE} and Ballmann-Brin-Spatzier \cite{BBS}, Ballmann and Burns-Spatzier proved the following higher rank rigidity theorem independently in \cite{Ba} and \cite{BS2}.

\begin{theorem} [Cf. \cite{Ba}, \cite{BS2}]
Let $M$ be a complete connected Riemannian manifold of finite volume and bounded nonpositive sectional curvature. Then the universal cover $\widetilde{M}$ of $M$ is a flat Euclidean space, a symmetric space of noncompact type, a space of rank 1 or a product of the above types.
\end{theorem}

If we assume that $\widetilde{M}$ is irreducible, i.e. not a Riemannian product of other two Riemannian manifolds of lower dimensions, we have:

\begin{theorem}[Cf. \cite{Ba}, \cite{BS2}]
Let $M$ be a complete connected Riemannian manifold of finite volume and bounded nonpositive sectional curvature with rank at least $2$, whose universal cover $\widetilde{M}$ is irreducible. Then $M$ is locally symmetric.
\end{theorem}

Our goal is to generalize the above higher rank rigidity theorem from Riemannian manifolds to a broader class, i.e. Finsler manifolds. A Finsler structure is given by a smooth family of Minkowski norms defined on each tangent space. Tensors such as the fundamental tensor, are defined on $TM$ instead of $M$, hence not only depend on the position $x$ on $M$, but also depend on the direction $y$ in the tangent space. Let $\pi:TM \to M$ be the natural projection. The pulled-back bundle $\pi^* TM$ admits a unique linear connection $\bigtriangledown$ which satisfies torsion freeness and almost $g$-compatibility, called Chern connection. The curvature $2$-form of Chern connection consists of two parts: $hh$- and $hv$-curvature parts. The flag curvature is then defined via the $hh$-curvature tensors, which coincides with the sectional curvature if the Finsler manifold is Riemannian. See Section $2$ for details. We want to prove the higher rank rigidity theorem for Finsler spaces of nonpositive flag curvature following the scheme in \cite{BBE}, \cite{BBS}, and \cite{BS2}. However, there are some technical difficulties we cannot overcome at this time, such as no  flat strip lemma, and no proper angle notion for general Finsler spaces.

In this paper we restrict to Berwald spaces, a special class of Finsler spaces larger than the class of Riemannian manifolds. In Berwald spaces, the Chern connection coefficients $\Gamma_{jk}^i$ have no $y$-dependence (See Section $2.4$). It follows that $hv$-curvature vanishes, and the behavior of the geodesics is well controlled by the nonpositivity of the flag curvature. More precisely, on the universal cover $\widetilde{M}$ of a Berwald space $M$, the nonpositivity of the flag curvature implies the convexity in $t$ of the distance function $d(\alpha(t), \beta(t)))$ between two geodesics $\alpha$ and $\beta$ (see Proposition \ref{distancefunction} below). Based on this convexity, we prove a flat strip lemma (see Lemma \ref{flatstripBerwald}), and thus the higher rank implies a great deal of flats. Then we construct the stable and unstable manifolds for the geodesic flows.  At last, we are able to construct Weyl Chambers and Tits building in the sphere at infinity for the universal cover of Berwald spaces with nonpositive flag curvature. To this end, we need overcome some difficulties causing by the Berwald metric. We define a notion of angle using the convexity (see Definition \ref{angledefinition}), which in turn defines a metric on $\widetilde{M}(\infty)$, the sphere at infinity. Furthermore, we obtain a type of coarse estimation on distance functions to deal with the issue of reference vector (see Proposition \ref{uniform}).

Therefore, we can go through all steps in \cite{BBE}, \cite{BBS}, and \cite{BS2}, and prove:

\begin{main}\label{higherrank}
Let $(M,F)$ be a complete connected Berwald space of finite volume and bounded nonpositive flag curvature with rank at least $2$, whose universal cover $\widetilde{M}$ is irreducible. Then $(M,F)$ is a locally symmetric space or a locally Minkowski space.
\end{main}




We note that our main theorem also deals with the irreversible case. A locally symmetric Finsler space is reversible unless it is an irreversible locally Minkowski space. It follows that a complete connected irreversible Berwald space of finite volume and bounded nonpositive flag curvature with rank at least $2$, whose universal cover $\widetilde{M}$ is irreducible, must be an (irreversible) locally Minkowski space.

Our proof of Main Theorem \ref{higherrank} in this paper is based on the study of the global properties of Berwald spaces with nonpositive flag curvature. We give a clear dynamical picture of the geodesic flow, and reveals a rigidity phenomenon that dynamical behaviour of the geodesic flow remains the same even in a broader geometric context.

The paper is organized as follows. In Section 2, we present some preliminaries on Finsler spaces and Berwald spaces, particularly on flag curvature and Jacobi fields. The key properties of Berwald spaces of nonpositive flag curvature are studied in Section 3. In Section 4, we prove Main Theorem \ref{higherrank} by adapting the methods in \cite{BBE}, \cite{BBS}, \cite{BS2}.

\section{Preliminaries on Finsler spaces}

In this section, we give definitions and present some basic results first on Finsler spaces, then on Berwald spaces. We will adapt the notations in \cite{BCS}.

\subsection{The Chern connection and flag curvature}
\begin{definition}
A Minkowski norm on an $n$-dimensional real vector space is a function $F:V\to [0,\infty)$ with the following properties:
\begin{enumerate}
\item $F$ is smooth on $V-\{0\}$;
\item $F(\lambda y)=\lambda F(y)$, $\forall \lambda >0$;
\item The Hessian matrix
$$(g_{ij}):=([\frac{1}{2}F^2]_{y^iy^j})$$
is positive definite at every point of $V-\{0\}$.
\end{enumerate}
\end{definition}

The simplest example of Minkowski norm is a Euclidean norm $F(y):=\sqrt{\langle y,y\rangle}, y\in \mathbb{R}^n$ where $\langle\cdot, \cdot\rangle$ denotes the canonical inner product. There are other examples, such as
$$F(y^1, y^2)=\sqrt{\sqrt{(y^1)^4+(y^2)^4}+\lambda((y^1)^2+(y^2)^2)},\quad y^1, y^2\in \mathbb{R}$$
where $\lambda$ can be any positive number.

A Finsler manifold is a manifold with a smoothly varying family of Minkowski norms, one on each tangent space. We give a precise definition of Finsler manifolds.

\begin{definition}
A Finsler manifold is a manifold $M$ with a $F:TM\rightarrow[0,\infty)$ (Finsler metric) such that:
\begin{enumerate}

\item $F$ is smooth on $TM\setminus\{0\}$;
\item $F\mid _{T_xM}: T_xM \rightarrow [0,\infty)$ is a Minkowski norm for all $x\in M$.
\end{enumerate}

\end{definition}

We recall some notions of tensors for a Finsler manifold $M$. In general, they are defined on the tangent bundle $TM$ instead of $M$, as distinct from the case for Riemannian manifolds. There is a natural coordinate system $(x^1,\cdots,x^n,y^1,\cdots,y^n)$ on $TM$, where $(x^1,\cdots,x^n)$ is a local coordinate system on $M$, and for any $y\in T_xM$, $y=y^j\frac{\partial}{\partial x^j}$. Here we list some tensors and quantities which will be useful later. We use Einstein summation convention throughout this paper.

\begin{enumerate}

\item The fundamental tensor $g=g_{ij}dx^i\otimes dx^j$ where:
$$g_{ij}(x,y)=(\frac{1}{2}F^2)_{y^iy^j}.$$

\item The Cartan tensor $C=C_{ijk}dx^i\otimes dx^j\otimes dx^k$ where:
$$C_{ijk}(x,y)=(\frac{1}{4}F^2)_{y^iy^jy^k}.$$
The Cartan tensor vanishes for Riemannian manifolds.

\item The formal Christofell symbols of the second type are
$$\gamma _{jk}^i(x,y)=\frac{1}{2}g^{is}(\frac{\partial g_{sj}}{\partial x^k}-\frac{\partial g_{jk}}{\partial x^s}+\frac{g_{ks}}{\partial x^j}),$$
where $(g^{ij})$ is the inverse matrix of $(g_{ij})$.

\item The nonlinear connection is the quantities
$$N_j^i(x,y):= \gamma_{jk}^iy^k-g^{il}C_{ljk}\gamma_{rs}^k y^ry^s,$$
which provides us a splitting of second tangent bundle $TTM$ into a horizontal subspace $H(TM)$ spanned by $\{\frac{\delta}{\delta x^j} \}$
and a vertical subspace $V(TM)$ spanned by $\{\frac{\partial}{\partial y^j}\}$, where:
$$\frac{\delta}{\delta x^j}:=\frac{\partial}{\partial x^j}-N_j^i\frac{\partial}{\partial y^i}.$$
We can also define the corresponding 1-forms $\{dx^i\}$ and $\{\delta y^i\}$ where:
$$ \delta y^i:=dy^i+N_j^i dx^j.$$

\end{enumerate}

The pulled-back tangent bundle $\pi ^{\ast}TM$ is induced by the projection $\pi: TM\to M$. It is a vector bundle over the slit tangent bundle $TM\setminus\{0\}$, with fiber over a typical point $(x,y)$ a copy of $T_xM$. The fundamental tensor and Cartan tensor are symmetric sections of $\pi^{\ast}T^{\ast}M\otimes \pi^{\ast}T^{\ast}M$ and $\otimes^3\pi^{\ast}T^{\ast}M$ respectively. Now the Chern connection is defined on $\pi^* TM$:

\begin{theorem}(Chern, cf. \cite{BCS})
The pulled-back bundle $\pi^* TM$ admits a unique linear connection $\bigtriangledown$ which satisfies torsion freeness and almost $g$-compatibility. In the natural coordinates,
$$\bigtriangledown _v \frac{\partial}{\partial x^j}:=\omega _j^i(v)\frac{\partial}{\partial x^i},$$
$$\omega _j^i =\Gamma_{jk}^i dx^k \ \text{\ and\ \ } \Gamma_{jk}^i=\Gamma_{kj}^i (\text{\ torsion freeness}),$$
and
$$ \Gamma_{jk}^l= \gamma_{jk}^{l}-g^{li}(C_{ijs}N_k^s-C_{jks}N_i^s+C_{kis}N_j^s)\ (\text{almost $g$-compatibility}).$$
\end{theorem}

The curvature 2-forms on $TM\setminus\{0\}$ of the Chern connection are
$$\Omega _j^i:= d\omega _j^i-\omega _j^k\wedge w_k^i. $$
As a priori, they can be expanded as $$\Omega_j^i=\frac{1}{2}R_{j\ kl}^i dx^k \wedge dx^l+P_{j\ kl}^i dx^k\wedge\frac{\delta y^l}{F}+\frac{1}{2}Q_{j\ kl}^i \frac{\delta y^k}{F}\wedge\frac{\delta y^l}{F}.$$
The tensors $R$,$P$,$Q$ are respectively the $hh$-, $hv$- and $vv$-curvature tensors of the Chern connection. It turns out that $Q_{j\ kl}^i=0$. Let $R_{jikl}=g_{is}R_{j\ kl}^s$. Now we can define the notion of flag curvature which is a generalization of the sectional curvature of Riemannian manifolds. Since the fundamental tensors and curvature tensors are defined on $TM\setminus\{0\}$, we need specify that $y\in T_xM$ as the flagpole and $V=V^i \frac{\partial}{\partial x^i}$ as a transverse edge. The flag curvature of the plane $\text{span}\{y,V\}$ is defined as:
\begin{equation}\label{e:flagcur}
K(y, V):= \frac{V^i(y^jR_{jikl}y^l)V^k}{g(y,y)g(V,V)-(g(y,V))^2}
\end{equation}
where
$$g:=g_{ij}(x,y)dx^i \otimes dx^j=(\frac{1}{2}F^2)_{y^iy^j}dx^i \otimes dx^j$$
is the lift of the fundamental tensor defined by the Hessian of the Finsler metric. When $M$ is a Riemannian manifold, the flag curvature defined above is exactly the sectional curvature.

\subsection{Geodesics and Jacobi fields}

Let $c:[a,b]\rightarrow M$ be a piecewise $C^1$ curve on a Finsler manifold $(M,F)$. The length of $c$ is defined as:
$$L(c):=\int_a^b F(c(t),c'(t))dt.$$
For a pair of points $p$ and $q$, we have an induced distance:
$$d(p,q)= \inf_c L(c)$$
where the infimum is taken among all the piecewise curves connecting $p$ and $q$. Generally it is not true that $d(p,q)=d(q,p)$ since the Finsler metric $F$ is defined only positively homogeneous.

Let $\sigma(t)$ be a smooth curve in $M$ with velocity field $T$. Let $W(t):=W^i(t)\frac{\partial}{\partial x^i}$ be a vector field along $\sigma$. We have two different covariant derivatives with respect to Chern connection according to the reference vectors:
\begin{enumerate}
\item with reference vector $T$:
 $$D_TW:= \left(\frac{dW^i}{dt}+W^jT^k(\Gamma _{jk}^i)_{(\sigma,T)}\right)\frac{\partial}{\partial x^i}\Big|_{\sigma(t)},$$
\item with reference vector $W$:
 $$D_TW:= \left(\frac{dW^i}{dt}+W^jT^k(\Gamma _{jk}^i)_{(\sigma,W)}\right)\frac{\partial}{\partial x^i}\Big|_{\sigma(t)}.$$
\end{enumerate}
A geodesic $\gamma$ on $M$ is a curve which locally minimizes its length. In natural coordinates, by considering the variation of arc length, a constant speed geodesic satisfies
$$D_{\gamma'}\gamma'=0$$
with reference vector $\gamma'$. Let $\sigma: T^1M \to M$ be the unit tangent bundle, i.e. $T^1M=\{(x,y)\in TM: F(x,y)=1\}$. The geodesic flow $g^t$ is defined on $T^1M$, whose orbits projecting to $M$ are unit speed geodesics. Its generator $X$ is a vector field on $T^1M$.

A Jacobi field along a geodesic is a variation vector field of geodesic variation. Let $J(t)$ be a Jacobi vector field along a unit speed geodesic $\gamma(t)$. We denote $T(t)=\gamma'(t)$ the velocity vector field along a geodesic $\gamma$, and $X$ the generator of geodesic flow on $T^1M$. Define
$$J'(t)= D_T J(t)$$ with reference vector $T$. Then $J(t)$ satisfies the Jacobi equation:
$$J''(t)+R(J,T)T=0$$ where $R$ is the curvature tensor related to the flag curvature (see \eqref{e:flagcur}), i.e.
$$R(J,T)T:= (T^jR_{j\ kl}^i T^l)J^k\frac{\partial}{\partial x^i}.$$

The geodesic flow $g^t$ can be described in a very similar way to the Riemannian case, due to the following Riemannian metric defined on $T^1M$. Recall the splitting of the second tangent bundle $TTM$ into a horizontal subbundle $H(TM)$ spanned by $\{\frac{\delta}{\delta x^j} \}$ and a vertical subbundle $V(TM)$ spanned by $\{\frac{\partial}{\partial y^j}\}$ (cf. \cite{BCS}). There is a Riemannian metric $g$ on the manifold $T^1M$ called the Sasaki metric
$$g_{ij}(x,y)dx^i \otimes dx^j+ g_{ij}(x,y)\frac{\delta y^i}{F}\otimes \frac{\delta y^j}{F}$$ such that:
\begin{enumerate}
\item the splitting $TT^1M=H(T^1M)\oplus V(T^1M) \oplus RX$ is orthogonal;
\item $g$ is well adapted to the pulled-back tangent bundle $\pi ^* T^1M$;
\item $g(X,X)=1$;
\item $g$ is $D _T$ invariant;
\item the curvature operator $R(T,\cdot)T$ is symmetric with respect to $g$.

\end{enumerate}
It is well known that geodesic flow preserves the volume form on $T^1M$ induced by the Sasaki metric, which is called the \emph{Liouville measure}. Similar to the Riemannian case, the tangent space of $T^1M$ can be identified with the space of Jacobi fields. We view the Jacobi field $J(t)$ as a section of $\pi^*TM \to T^1M$. So by the second property above, $g(J(t), J(t))= g_T(J,J)$ where $g_T:= g_{ij}(\sigma, T) dx^i \otimes dx^j$. The differential of the geodesic flow can be described via Jacobi fields:
\begin{proposition}(Cf. Appendix in \cite{Fo})\label{Foulon}
Suppose for $z\in T^1M$, $Z \in T_zT^1M$, we have the splitting $dg^t Z= H_t + V_t+ a_t X$, where $H_t \in H(T^1M)$ and $V_t \in V(T^1M)$. Then there exist isometries
$$v_X: \pi ^* T^1M \to V(T^1M) \text{\ and\ } h_X: \pi ^* T^1M \to H(T^1M)$$
such that $h_X^{-1}(H_t)= J(t)$ and $v_X^{-1}(V_t)=J'(t)$, and:
$$g(dg^t Z, dg^t Z)= g(J(t),J(t))+ g(J'(t),J'(t))+g(a_t X, a_t X).$$
\end{proposition}

\subsection{Berwald spaces}

A Finsler space $(M,F)$ is said to be a Berwald space if the Chern connection coefficients $\Gamma_{jk}^i$ have no $y$ dependence in natural coordinates, in other words, the Chern connection is defined directly on the underlying manifold $M$. As a result, the $hv$-curvature tensor $P_{j\ kl}^i$ vanishes and only $hh$-curvature tensor remains. In this sense, Berwald spaces are just a little bit more general than Riemannian spaces.

Recall the covariant derivative of $W$ along a curve $\sigma$ on $M$:
$$ D_TW= \left(\frac{dW^i}{dt}+W^jT^k \Gamma _{jk}^i\right)\frac{\partial}{\partial x^i}\Big|_{\sigma(t)}.$$
So when we deal with covariant derivative for a Berwald space, we don't have an issue of reference vector.

Ichijy\={o} proved that for Berwald spaces all the tangent spaces are linearly isometric to a common Minkowski space. The proof below is taken from \cite{BCS}.

\begin{proposition}(Ichijy\={o}, cf.\cite{BCS}) \label{isom}
Let $(M,F)$ be a connected Berwald space. Then its tangent spaces $(T_x M, F_x)$ with associated Minkowski norms, are linearly isometric to each other by parallel translations.
\end{proposition}

\begin{proof}
Take any two points $p$, $q$ on $M$ and a curve $\sigma$ (with velocity field $T$) connecting them. Let $W$ be a parallel vector field, i.e. $D_T W=0$ (reference vector is irrelevant here). It is linear in $W$ since $\Gamma_{jk}^i$ is only depend on $\sigma (t)$. So parallel translations define a linear mapping from $T_pM$ to $T_qM$ and obviously one to one. It is enough to prove that the Finsler norm is preserved under parallel translations. Recall that $F(W)= \sqrt{g_W(W,W)}$ where $g_W:=g_{ij}(\sigma, W)dx^i\otimes dx^j$.

The almost $g$-compatibility criterion of the Chern connection is
$$dg_{ij}-g_{kj}\omega_i^k-g_{ik}\omega_j^k=2C_{ijs}\delta y^s.$$
Evaluate it at $(\sigma,W)(t)$, then contract with $W^iW^j$. By Euler's theorem, the right side is zero since $C_{ijk}$ is $(-1)$-homogeneous in $y$ variable. So
$$W^iW^j(dg_{ij}-g_{kj}\Gamma_{is}^kdx^s-g_{ik}\Gamma_{js}^kdx^s)=0.$$
Contract this $1$-form equation with the velocity of the lift $(\sigma,W)$:
$$W^iW^j(\frac{dg_{ij}}{dt}-g_{kj}\Gamma_{is}^kT^s-g_{ik}\Gamma_{js}^kT^s)=0.$$
Simplifying it we have
$$\frac{d}{dt}g_W(W,W)=2 g_W(D_T W,W).$$
Hence the Finsler norm is preserved under parallel translations.
\end{proof}

Moreover, there are Riemannian isometries between two punctured tangent spaces, and hence between the two unit spheres in tangent spaces. For each fixed $x \in M$, the tensor $\hat{g}:= g_{ij}(x,y) dy^i \otimes dy^j$ defines a smooth Riemannian metric $\hat{g}_x$ on each punctured space $T_xM \setminus 0$. Hence $T^1_xM$ also has an induced Riemannian metric, say $\dot{g}_x$.

\begin{corollary}\label{isometries}
Let $(M,F)$ be a connected Berwald space, $p,q \in M$, and $\sigma(t)$ be a curve on $M$ connecting $p$ and $q$. Parallel translations along $\sigma$ induce Riemannian isometries $(T_pM\setminus 0, \hat{g}_p) \rightarrow (T_qM\setminus 0, \hat{g}_q)$ and $(T^1_pM,\dot{g}_p)\to (T^1_qM, \dot{g}_q)$.
\end{corollary}

\begin{proof}
Let $\phi: T_pM \rightarrow T_qM$ be the parallel translation. By Proposition \ref{isom} it is an linear isometry between the two Minkowski spaces. Let $\{u_1,\ldots,u_n\}$ and $\{v_1,\ldots, v_n\}$ be the basis for $T_pM$ and $T_qM$ respectively. Then there exists a matrix $A=(a_i^j)$ such that $\phi(u_i)= a_i^j v_j$. Since $\phi$ is linear, it is enough to prove $\hat{g}_W(U,V)=\hat{g}_{\phi(W)}(\phi(U),\phi(V))$ for any $U,V,W \in T_pM$.

Since $\phi$ preserves Minkowski norm, we have
$$\frac{1}{2}F_p^2(W^1,\ldots,W^n)=\frac{1}{2}F_q^2(a_j^1 W^j,\ldots,a_j^n W^j).$$
After taking second derivative, we have
\begin{equation}\label{e:isometry}
g_{ij}(\sigma,W)= a_i^kg_{kl}(\sigma,\phi(W))a_j^l
\end{equation}
Also $(\phi(U))^k=a_i^kU^i$, $(\phi(V))^l=a_i^lV^i$. By \eqref{e:isometry}, we have $\hat{g}_W (U,V)=\hat{g}_{\phi(W)}(\phi(U),\phi(V))$. Hence $\phi:(T_pM\setminus 0, \hat{g}_p) \rightarrow (T_qM\setminus 0, \hat{g}_q)$  is a Riemannian isometry.
\end{proof}

\section{Berwald spaces of nonpositive flag curvature}

In the rest of this paper, let $(M,F)$ be a complete connected Berwald space with flag curvature $K\leq 0$ and all the geodesics are parameterized by arc length. In this case, Cartan-Hadamard theorem holds and $\widetilde{M}$, the universal cover of $M$, is diffeomorphic to $\mathbb{R}^n$. In this section, we study some properties of Berwald spaces due to nonpositive flag curvature which will play important roles in the proof of higher rank rigidity theorem.

\subsection{Convexity}
On Riemannian manifolds, the nonpositive curvature implies the convexity of the length of Jacobi fields $\|J(t)\|$. It follows that on the universal cover $\widetilde{M}$, the distance function $d(\gamma_1(t), \gamma_2(t))$ is convex in $t$. This is not true generally for Finsler spaces of nonpositive flag curvature. Nevertheless, for Berwald spaces of nonpositive flag curvature, convexity of the distance function follows from a result in \cite{KK} due to A. Krist\'{a}ly and L. Kozma. We repeat the proof here, and the argument is used several times later, especially in the proof of Lemma \ref{flatstripBerwald}, i.e. the flat strip lemma for Berwald spaces. Since the reference vector is irrelevant here in a Berwald space, we have the following nice product rule:
$$\frac{d}{dt}g_W(U,V)=g_W(D_TU, V)+g_W(U, D_TV)$$
when $U$ or $V$ is proportional to $W$.
\begin{proposition}\label{distancefunction}
Let $\alpha(t)$,$\beta(t)$ be two geodesics on $\widetilde{M}$, then the distance function $d(\alpha(t),\beta(t))$ is convex in $t$.
\end{proposition}
\begin{proof}
First we prove for the special case when geodesics $\alpha(t)$ and $\beta(t)$ emanate from a point $p$. We want to prove $d(\alpha(\frac{T}{2}),\beta(\frac{T}{2}))\leq \frac{1}{2}d(\alpha(T),\beta(T))$, for any $T \in \mathbb{R}$.

Let $\gamma:[0,S]\rightarrow \widetilde{M}$ be the unique geodesic connecting $\alpha(T)$ and $\beta(T)$. Define $\Sigma:[0,T]\times [0,S]\rightarrow \widetilde{M}$ by
$$\Sigma(t,s)= \exp_{\gamma(s)}((1-\frac{t}{T})\exp^{-1}_{\gamma(s)}(p)).$$
Then $\Sigma$ is a geodesic variation with $\Gamma(\cdot,0)=\alpha$ and $\Gamma(\cdot,S)=\beta$. $J_s(t):=\frac{\partial}{\partial s}\Sigma(t,s)$ is a Jacobi field along $t$-curve. Since flag curvature $K\leq 0$, there is no conjugate point. $J_s(0)=0$, So $J_s(t)\neq 0$ for $t\in(0,T]$. Let $T_s$ be the velocity field of $\Sigma(\cdot, s)$.
Using the product rule and $F(J_s)=[g_{J_s}(J_s,J_s)]^{\frac{1}{2}}$, we have:
\begin{equation}\label{e:long}
\begin{aligned}
&\frac{d^2}{dt^2}F(J_s)=\frac{d}{dt}[\frac{g_{J_s}(D_{T_s}J_s,J_s)}{F(J_s)}]=
\\&\frac{g_{J_s}(D_{T_s}D_{T_s}J_s,J_s)F^2(J_s)+ g_{J_s}(D_{T_s}J_s,D_{T_s}J_s)F^2(J_s)-g^2_{J_s}(D_{T_s}J_s,J_s)}{F^3(J_s)}.
\end{aligned}
\end{equation}
By Jacobi equation, formula for flag curvature and Schwarz inequality, the first term in the numerator of \eqref{e:long} is
\begin{equation*}
\begin{aligned}
g_{J_s}(D_{T_s}D_{T_s}J_s,J_s)&=-g_{J_s}(R(J_s,T_s)T_s,J_s)\\&=-g_{J_s}(R(T_s,J_s)J_s,T_s)\\
&=-K(J_s,T_s)\cdot[g_{J_s}(J_s,J_s)g_{J_s}(T_s,T_s)-
g_{J_s}^2(J_s,T_s)]
\\&\geq 0.
\end{aligned}
\end{equation*}
The last two terms in the numerator of \eqref{e:long} are also nonnegative by Schwarz inequality. Hence
\begin{equation}\label{e:convexform}
\begin{aligned}
\frac{d^2}{dt^2}F(J_s)(t)\geq 0, \text{\ for all\ } \ t\in (0,T].
\end{aligned}
\end{equation}
$F(J_s)(t)$ is $C^{\infty}$ on $(0,1]$ and continuous on $[0,1]$, so by \eqref{e:convexform},
\begin{equation*}
F(J_s)(\frac{T}{2})\leq \frac{1}{2}F(J_s)(T).
\end{equation*}
Hence
\begin{equation}\label{e:convexform1}
\begin{aligned}
&d(\alpha(\frac{T}{2}),\beta(\frac{T}{2}))\leq \int_{0}^{S}F(J_s)(\frac{T}{2})ds \\
\leq &\frac{1}{2}\int_0^S F(J_s)(T)ds = \frac{1}{2}d(\alpha(T),\beta(T)).
\end{aligned}
\end{equation}

We consider another special case when $\alpha(T)=\beta(T)$ for some $T>0$. A construction similar to the one above shows that
\begin{equation*}
F(J_s)(\frac{T}{2})\leq \frac{1}{2}F(J_s)(0),
\end{equation*}
hence
\begin{equation}\label{e:convexform11}
\begin{aligned}
d(\alpha(\frac{T}{2}),\beta(\frac{T}{2}))\leq \frac{1}{2}d(\alpha(0),\beta(0)).
\end{aligned}
\end{equation}
For the general case, let $\gamma$ be the geodesic segment connecting $\beta(t_1)$ and $\alpha(t_2)$ and $q$ be the midpoint of $\gamma$. Then
\begin{equation}\label{e:convexform2}
\begin{aligned}
d(\alpha (\frac{t_1+t_2}{2}),\beta(\frac{t_1+t_2}{2}))&\leq d(\alpha (\frac{t_1+t_2}{2}),q)+d(q,\beta (\frac{t_1+t_2}{2}))\\
&\leq \frac{1}{2}d(\alpha(t_1),\beta(t_1))+\frac{1}{2}d(\alpha(t_2),\beta(t_2)).
\end{aligned}
\end{equation}
The second inequality in \eqref{e:convexform2} follows from \eqref{e:convexform1} and \eqref{e:convexform11}. As the distance function is continuous, the $\frac{1}{2}$-convexity implies convexity.
\end{proof}

A geodesic space with the convexity property \eqref{e:convexform1} above is said to be of \emph{nonpositive curvature in the sense of Busemann} (globally in our case since $\widetilde{M}$ is simply connected). It is conjectured that Finsler manifolds of nonpositive curvature in sense of Busemann must be of Berwald type, see \cite{KK}.

Nonpositive curvature in the sense of Busemann is a weaker notion than \emph{nonpositive curvature in the sense of Aleksandrov}, see \cite{J} for definitions. A Berwald space of nonpositive flag curvature is not necessarily of nonpositive curvature in the sense of Aleksandrov, see \cite{KK}. In fact, if a reversible Finsler manifold is of nonpositive curvature in the sense of Aleksandrov, then it must be a Riemannian manifold of nonpositive sectional curvature. Hence we cannot define the Aleksandrov angle for a Berwald space of nonpositive flag curvature. Nevertheless, we still can define a notion of angle in the next subsection. This angle notion is enough to play the role of angles in Riemannian manifolds, in the proof of higher rank rigidity theorem.
\subsection{A notion of Angle}

Based on the convexity property, we define a notion of angle to measure the distance between two directions. Consider a point $p \in \widetilde{M}$, and two geodesics $c_1(t)$ and $c_2(t)$ emanating from $p$. By convexity, the function $\frac {d(c_1(t),c_2(t))}{t}$ is nondecreasing in $t$. So the limit
$$\lim_{t\rightarrow 0}\frac {d(c_1(t),c_2(t))}{t}$$
exists. By triangle inequality of $d$, the limit is bounded by $2$.
\begin{definition}\label{angledefinition}
Let $u,v \in T^1_p\widetilde{M}$, $c_1(t)$ and $c_2(t)$ are two geodesics emanating from $p$ such that $c_1'(0)=u$ and $c_2'(0)=v$. The angle between $u$ and $v$ is defined as:
$$\angle_p(u,v):=\lim_{t\rightarrow 0}\frac {d(c_1(t),c_2(t))}{t}.$$
\end{definition}

The tangent space $T_p \widetilde{M}$ equipped with Minkowski norm $F_p:=F(p,\cdot)$ is a Finsler manifold if we identify $T_y(T_p \widetilde{M})$ with $T_p \widetilde{M}$ itself and provide it the norm $F(p,\cdot)$. We also call this Finsler manifold Minkowski space. In Minkowski space, the Finsler metric $F$ has no $x$ dependence in natural coordinates. It turns out that Chern connection coefficient $\Gamma_{ij}^k=0$, and hence the flag curvature is always zero. Moreover, the geodesics in Minkowski space are lines. In Minkowski space $(T_p \widetilde{M},F_p)$, let $\tilde{c}_1(t)$ and $\tilde{c}_2(t)$ be the two geodesic rays emanating from origin in the direction $u$ and $v$. Then:

\begin{proposition}\label{angle}
$$\lim_{t\rightarrow 0}\frac {d(c_1(t),c_2(t))}{t}= \lim_{t\rightarrow 0}\frac {d(\tilde{c}_1(t),\tilde{c}_2(t))}{t}.$$
\end{proposition}

\begin{proof}
For Berwald space, the exponential map $\exp_p$ is $C^{\infty}$ on $T_p \widetilde{M}$ and $d \exp_p|_p=Id$, see \cite{BCS}.

Consider $\tilde{\Sigma}: [0,1]\times [0,1]\rightarrow T_p\widetilde{M}$ defined by:
$$\tilde{\Sigma}(s,t)=t[(1-s)u+sv].$$
Let $\Sigma(s,t)=\exp_p (\tilde{\Sigma}(s,t))$ defined on $\widetilde{M}$.
\begin{equation}\label{e:angle}
\begin{aligned}
\lim_{t\rightarrow 0} \frac{d(c_1(t),c_2(t))}{t}&\leq \lim_{t\rightarrow 0}\frac{\int_0^1 F(\Sigma(s,t), \frac{\partial}{\partial s}\Sigma(s,t))ds}{t}\\
&=\lim_{t\rightarrow 0}\int_0^1 F(\Sigma(s,t), \frac{1}{t}d \exp_p|_{\tilde{\Sigma}(s,t)}[\frac{\partial}{\partial s}\tilde{\Sigma}(s,t)])ds\\
&=\lim_{t\rightarrow 0}\int_0^1 F(\Sigma(s,t), d \exp_p|_{\tilde{\Sigma}(s,t)}(v-u))ds\\
&=\int_0^1 F(p,v-u) ds\\
&=\lim_{t\rightarrow 0}\frac {d(\tilde{c}_1(t),\tilde{c}_2(t))}{t}.
\end{aligned}
\end{equation}
We used continuity of $d\exp_p$ and $F$ in the third equality in \eqref{e:angle}.

Conversely, consider $\sigma: [0,1]\times [0,1]\rightarrow \widetilde{M}$ defined by:
$$\sigma(s,t)=\exp_{c_1(t)}[s\exp^{-1}_{c_1(t)}c_2(t)].$$
Let $\tilde{\sigma}(s,t)=\exp^{-1}_p [\sigma(s,t)]$ defined on $T_p \widetilde{M}$.
\begin{equation}\label{e:angle0}
\begin{aligned}
\lim_{t\rightarrow 0} \frac{d(c_1(t),c_2(t))}{t}&= \lim_{t\rightarrow 0}\frac{\int_0^1 F(\sigma(s,t), \frac{\partial}{\partial s}\sigma(s,t))ds}{t}\\
&=\lim_{t\rightarrow 0} \int_0^1 F(\sigma(s,t), \frac{1}{t}d \exp_p|_{\tilde{\sigma}(s,t)}[\frac{\partial}{\partial s}\tilde{\sigma}(s,t)])ds\\
&=\lim_{t\rightarrow 0}\int_0^1 F(p, \frac{1}{t}\frac{\partial}{\partial s} \tilde{\sigma}(s,t))ds\\
&\geq\lim_{t\rightarrow 0}\frac {d(\tilde{c}_1(t),\tilde{c}_2(t))}{t}.
\end{aligned}
\end{equation}
At the third equality in \eqref{e:angle0}, we used the smoothness of $\exp_p$ (at least $C^2$) and smoothness of $F$ (at least $C^1$) away from zero section.
\end{proof}

\begin{corollary}
Let $u,v \in T^1_p\widetilde{M}$. Then $\angle_p(u,v)=F(p,u-v)$.
\end{corollary}
\begin{proof}
This is an immediate corollary of Proposition \ref{angle} since $\lim_{t\rightarrow 0}\frac {d(\tilde{c}_1(t),\tilde{c}_2(t))}{t}=F(p,u-v)$ in Berwald space $\widetilde{M}$.
\end{proof}
\begin{corollary}
$d(c_1(t),c_2(t))\geq d(\tilde{c}_1(t),\tilde{c}_2(t)).$
\end{corollary}
\begin{proof}
This follows from Proposition \ref{angle} and that $\frac{1}{t}d(c_1(t),c_2(t))$ is nondecreasing in $t$, but $\frac{1}{t}d(\tilde{c}_1(t),\tilde{c}_2(t))$ is a constant.
\end{proof}

Proposition \ref{angle} implies that the angle is indeed a metric on $T_p^1\widetilde{M}$ and it induces the usual topology on $T_p^1\widetilde{M}$ for any $p\in \widetilde{M}$. In particular, the angle metric is complete on $T^1_p\widetilde{M}$.

\subsection{Sphere at infinity}

Two geodesics $\gamma$ and $\delta$ are called to be asymptotes if $d(\gamma(t),\delta(t)) \leq c$ for some constants $c$ and $\forall t\geq 0$. We also call $u,v \in T^1\widetilde{M}$ asymptotic if $\gamma_u$ and $\gamma_v$, the geodesics with initial vector $u$ and $v$ respectively, are asymptotic.

\begin{lemma}
Let $c_1(t)$, $c_2(t)$ be two distinct geodesics emanating from $p \in \widetilde{M}$. Then
$$d(c_1(t),c_2(t))\rightarrow \infty \ \  \text{for}\  t\rightarrow \infty.$$
\end{lemma}

\begin{proof}
It follows from the convexity of function $d(c_1(t), c_2(t))$.
\end{proof}

Similar to the Riemannian case, we have
\begin{proposition}\label{asymptotic}
Let $\gamma(t)$ be a geodesic on $\widetilde{M}$. For any $p \in \widetilde{M}$, there exists a unique geodesic starting at $p$ and asymptotic to $\gamma$.
\end{proposition}

\begin{proof}
Same proof as in Proposition 1.2 in \cite{EO}. Convexity of the distance function $d(\gamma(t),\delta(t))$ is used in the uniqueness part.
\end{proof}

The asymptotes relation is an equivalence relation. Denote by $\widetilde{M}(\infty)$ the set of all equivalence classes. By Proposition \ref{asymptotic}, there exists a unique geodesic, denoted by $\gamma_{px}$, connecting $p\in \widetilde{M}$, $x \in \widetilde{M}(\infty)$. So we can define a topology on $\widetilde{M}(\infty)$ such that the map $x\rightarrow \gamma_{px}'(0)$ is a homeomorphism between $\widetilde{M}(\infty)$ and $T^1_p\widetilde{M}$. This topology on $\widetilde{M}(\infty)$ is independent of $p$ and we call $\widetilde{M}(\infty)$ the sphere at infinity. Moreover, for any $p\in \widetilde{M}$, $x,y \in \widetilde{M}\cup\widetilde{M}(\infty)$,($p\neq x, p\neq y$) we can define $\angle_p(x,y)=\angle_p(\gamma'_{px}(0),\gamma'_{py}(0))$. The angle $\angle_p(x,y)$ depends continuously on $p$, $x$, and $y$.

Two geodesics $\alpha$ and $\beta$ are said to be parallel if $d(\alpha(t),\beta(t))\leq c$ for some constant $c$ and $\forall t \in \mathbb{R}$. In fact by convexity, for parallel geodesics $\alpha$ and $\beta$, $d(\alpha(t),\beta(t))\equiv c$ for some constant $c$ and $\forall t \in \mathbb{R}$. We call $u$, $v$ parallel if $\gamma_u$ and $\gamma_v$ are parallel geodesics. $u\parallel v$ if and only if $u$ is asymptotic to $v$ and $-u$ is asymptotic to $-v$.

\begin{lemma}\label{angleestimate}
Let $\gamma$ be an arbitrary geodesic in $\widetilde{M}$, and $\gamma'(0)=v$. Let $x\in \widetilde{M}(\infty)$. For any $t\geq 0$, let $w(t)$ be the unique vector at point $\gamma(t)$ such that $\gamma_{w(t)}(\infty)=x$. Then:
$$\angle_{\gamma(0)}(v, w(0))\leq \angle_{\gamma(t)}(g^t(v), w(t))$$
for any $t>0$.
\end{lemma}

\begin{proof}
We fix $t>0$, and let $\alpha(s)$, $\beta(s)$ be the geodesics emanating from $\gamma(0)$ and $\gamma(t)$ respectively, such that $\alpha(\infty)=\beta(\infty)=x$. Let $s>0$. Denote $d(t)=d(\gamma(t), \alpha(t))$, $d(s+t)=d(\gamma(s+t), \alpha(s+t))$, $d_1(s)=d(\gamma(t+s), \beta(s))$, $c(s,t)=d(\alpha(t+s), \beta(s))$.

By triangle inequality, $d(s+t)\leq c(s,t)+d_1(s)$. Since $\alpha$ and $\beta$ are asymptotic, the distance function between $\alpha$ and $\beta$ (with a time shift $t$) is decreasing. Thus $c(s,t)\leq d(t)$. We have:
$$\frac{d(t)}{t}\leq \frac{d(s+t)}{s+t}\leq \frac{c(s,t)+d_1(s)}{s+t}\leq \frac{d(t)+d_1(s)}{s+t}.$$
Hence $d(t)(s+t)\leq t(d(t)+d_1(s))$ which gives $\frac{d(t)}{t}\leq \frac{d_1(s)}{s}$. Let $s\rightarrow 0$, then we have
$$\frac{d(t)}{t}\leq \angle_{\gamma(t)}(g^tv, w(t)).$$
But by convexity, $\angle_{\gamma(0)}(u, w(0))= \lim_{r\rightarrow 0} \frac{d(r)}{r} \leq \frac{d(t)}{t}$. Hence
$$\angle_{\gamma(0)}(v, w(0))\leq \angle_{\gamma(t)}(g^tv, w(t)).$$
\end{proof}

The equality holds in Lemma \ref{angleestimate} if and only if the two geodesics $\gamma$ and $\alpha$ bound an area which is totally geodesic and flat. In fact we can also prove for any geodesic triangle in $\widetilde{M}$ an exterior angle is larger than the corresponding interior angle by a similar proof, then Lemma \ref{angleestimate} also follows from this fact and the continuity of angle functions. Lemma \ref{angleestimate} will be used in the proof of higher rank rigidity theorem.

Now we can define a (Tits) metric on $\widetilde{M}(\infty)$. Let $x, y \in \widetilde{M}(\infty)$. Choose arbitrary $p\in \widetilde{M}$, and let $\alpha$ and $\beta$ be the two geodesics emanating from $p$ such that $\alpha(\infty)=x$ and $\beta(\infty)=y$. Denote $d(t)=d(\alpha(t), \beta(t))$.
\begin{definition}
Define
$$d(x,y)=\lim_{t\rightarrow \infty}\frac{d(t)}{t}.$$
\end{definition}

It is easy to prove that limit above is independent of the choice of $p$, and $d(x,y)$ is indeed a metric on $\widetilde{M}(\infty)$.

\subsection{Flat strip lemma}

A flat strip means a totally geodesic isometric imbedding $r:\mathbb{R}\times [0,c]\rightarrow \widetilde{M}$, where $\mathbb{R}\times [0,c]$ is a strip in a Minkowski plane.

\begin{lemma}\label{flatstripBerwald}
If two distinct geodesics $\alpha$ and $\beta$ are parallel, then they bound a flat strip in $\widetilde{M}$.
\end{lemma}

\begin{proof}
Since geodesic $\beta(t), t \in \mathbb{R}$ is a convex set, there exists a unique foot point of $\alpha(0)$ on $\beta$. We may suppose that $\beta$ is parameterized such that $\beta(0)$ is the foot point of $\alpha(0)$ on $\beta$. Since $\alpha$ and $\beta$ are parallel, for any $a \in \mathbb{R}$, $d(\alpha(t), \beta(a+t))\equiv\ c$. But $d(\alpha(0),\beta(a))\geq d(\alpha(0),\beta(0))$, so $d(\alpha(t), \beta(a+t))\geq d(\alpha(t),\beta(t))$ for any $a\in \mathbb{R}$, i.e. $\beta(t)$ is the foot point of $\alpha(t)$ on geodesic $\beta$. Furthermore, it also follows that  $\alpha(t)$ is the foot point of $\beta(t)$ on geodesic $\alpha$. Choose arbitrary $t_1, t_2 \in \mathbb{R}$ and let $\gamma(s)$ and $\delta(s)$ be the geodesics connecting $\alpha(t_1)$ and $\beta(t_1)$, $\alpha(t_2)$ and $\beta(t_2)$ respectively. Denote $d(\alpha(t),\beta(t))\equiv c$. Consider $\Sigma: [t_1, t_2]\times [0,c]\rightarrow \widetilde{M}$ defined by
$$\Sigma(t,s)=\exp_{\gamma(s)}\left(\frac{t-t_1}{t_2-t_1}\cdot \exp^{-1}_{\gamma(s)}\delta(s)\right).$$
Then $\Sigma$ is a geodesic variation with all $t$-curves geodesics.

We first prove that all $s$-curves are also geodesics. Let $J_s(t)$ be the variation vector field, then same computation as in Proposition \ref{distancefunction}, we have:
$$\frac{d^2}{dt^2}F(J_s)(t)\geq 0, \ \text{for all} \ t\in [t_1,t_2].$$
So length of the $s$-curves $L(\Sigma(t,\cdot))$ is convex in $t$. But $L(\Sigma(t_1,\cdot))= L(\Sigma(t_2,\cdot))=c$, hence $L(\Sigma(t,\cdot))=c=d(\alpha(t),\beta(t))$ for all $t\in[t_1,t_2]$. So all $s$-curves are geodesics. It also follows that $\frac{d^2}{dt^2}F(J_s)(t)= 0$, for all $\ t\in [t_1,t_2]$, hence flag curvature $K=0$ on $\Sigma([t_1, t_2]\times [0,c])$ for any $t_1 <t_2$.

Next we prove $\Sigma([t_1, t_2]\times [0,c])$ is totally geodesic. It is enough to prove geodesic segment connecting any two points $\Sigma(t',s')$ and $\Sigma(t'',s'')$ lies in $\Sigma([t_1, t_2]\times [0,c])$. In fact it is enough to prove the geodesic segment $\kappa(t)$ connecting $\alpha(t_1)$ and $\beta(t_2)$ lie in $\Sigma([t_1, t_2]\times [0,c])$. Suppose $\kappa(t)$ has speed $\frac{d(\alpha(t_1), \beta(t_2))}{t_2-t_1}$. Since $\kappa(t)$ and $\alpha(t)$ are two geodesics emanating from $\alpha(t_1)$, we can construct a geodesic variation $\Sigma_1$ as in Proposition \ref{distancefunction}, and it follows that $L(\Sigma_1(t,\cdot))\leq \frac{t-t_1}{t_2-t_1}\cdot c$. Similarly, $\beta$ and $\kappa$ are two geodesics merging at $\beta(t_2)$ and we can construct geodesic variation $\Sigma_2$ such that $L(\Sigma_2(t,\cdot))\leq \frac{t_2-t}{t_2-t_1}\cdot c$. Hence $L(\Sigma_1(t,\cdot))+L(\Sigma_2(t,\cdot))\leq \frac{t-t_1}{t_2-t_1}\cdot c+\frac{t_2-t}{t_2-t_1}\cdot c=c=d(\alpha(t), \beta(t))$. It follows that the joining of $\Sigma_1(t,\cdot)$ and $\Sigma_2(t,\cdot)$ is the geodesic connecting $\alpha(t)$ and $\beta(t)$ and hence $\kappa(t)$ lie on this geodesic. Hence $\kappa(t)$ lies in $\Sigma([t_1, t_2]\times [0,c])$.

Since $\Sigma([t_1, t_2]\times [0,c])$ is totally geodesic and flag curvature $K=0$, it follows that $\Sigma([t_1, t_2]\times [0,c])$ is an imbedded Minkowski rectangle. Since $t_1$ and $t_2$ are arbitrary, the lemma follows.
\end{proof}

\section{Higher rank rigidity}

We shall prove the higher rank rigidity theorem for Berwald spaces in this section. Let $(\widetilde{M}, F)$ always be a complete, simply connected Berwald space with flag curvature $K\leq 0$.

Moreover, if not pointing out specifically, we suppose that $\widetilde{M}$ admits a quotient manifold $M$ of finite volume, and the flag curvature is bounded from below $-b^2 \leq K \leq 0$. The finiteness of volume of $M$ guarantees the finiteness of Liouville measure of $T^1M$.

Volume of a Finsler manifold can be defined as follows. Since the Sasaki metric is a Riemannian metric on $TM$, it induces a volume form on $TM$:
$$dV_g=\det(g_{ij}(x,y))dx^1\cdots dx^ndy^1\cdots dy^n,$$
which gives the Liouville measure on $T^1M$ preserved by the geodesic flow. Hence we can define a Finsler volume form on $M$ by
$$dV_{F}:=\sigma_{F}(x)dx^1\cdots dx^n,$$
where $\sigma_{F}(x):= \frac{\int_{B_x^n}\det(g_{ij}(x,y))dy^1\cdots dy^n}{\text{Vol}(\mathcal{B}^n)}$, $B_x^n:=\{y\in T_xM| F(y)\leq 1\}$, and $\mathcal{B}^n$ is a Euclidean unit ball. In such a way we have:
$$\int_{T^1M}fdV_g=\int_M dV_F(x)\int_{T^1_xM}f|_{T^1_xM}dS^{n-1},$$
where $dS^{n-1}$ is the Euclidean volume element on $(n-1)$-sphere. The above defined volume measure $V_F$ is called the \emph{Sasaki volume}. It is also discovered by R.D. Holmes and A.C. Thompson by Minkowski geometry approach. Consequently, it is also called \emph{Holmes-Thompson volume}. There are other notions of volume such as Busemann-Hausdorff volume. The Sasaki volume has the following advantage: it is easy to see that the Liouville measure is finite for a Finsler manifold of finite Sasaki volume. See Section 5.1 in \cite{Sh} for various notions of volume.

From now on, we will only consider the nontrivial case when $(\widetilde{M},F)$ is not a Minkowski space, i.e. $\text{rank}(\widetilde{M})<\dim \widetilde{M}$. We follow the structure of \cite{BBE}, \cite{BBS}, and \cite{BS2}. Firstly in subsection 4.1, we give the definition of rank and study the properties of Jacobi fields for Berwald spaces. In subsection 4.2, we prove that a higher rank Berwald space admits a great deal of $k$-flats, that is, complete, flat, totally geodesic $k$-dimensional submanifolds without boundary where $k=\text{rank}(\widetilde{M})$. In the third subsection, we construct strong stable manifolds for regular vectors. Next in subsection 4.4, we construct Weyl chambers, $(k-1)$ first integrals, and prove a closing lemma. Following that, we prove that all Weyl chambers are isometric to each other and can be extended to $\widetilde{M}(\infty)$, the boundary of $\widetilde{M}$ at infinity. In subsection 4.5, we prove that $\widetilde{M}(\infty)$ has a structure of Tits building. And the higher rank rigidity theorem follows from a similar argument of Gromov in the last subsection 4.6.

We try to keep the convention of notations in \cite{BBE}, \cite{BBS}, and \cite{BS2}. We mainly focus on the difference between the Berwald case and the Riemannian case, but just simply give references to \cite{BBE}, \cite{BBS}, and \cite{BS2} if the proof of a result can be extended verbatim to the Berwald case.

\subsection{Rank and more on Jacobi fields}

As in the Riemannian case, a parallel Jacobi field $J(t)$ along a geodesic $\gamma$ (with velocity field $T$) satisfies $D_TJ=0$ with reference vector $T$. But since $M$ is a Berwald space, the reference vector is irrelevant here.

\begin{definition}
Let $v \in T^1M$. We define rank$(v)$ as the dimension of the space of parallel Jacobi fields along geodesic $\gamma_v$, and $\text{rank}(M):=\min_{v\in T^1M} \text{rank} (v)$.
\end{definition}

The operator $R(T,\cdot)T$ is symmetric and negatively semidefinite with respect to $g_T:= g_{ij}(\sigma, T)dx^i\otimes dx^j$. Hence Jacobi field $J(t)$ is parallel if and only if $K(T, J)=0$. So rank of a vector reflects infinitesimal flatness along the geodesic. The notions above can be extended to universal cover space $\widetilde{M}$.

By a limit argument, $\text{rank}(w)\leq \text{rank}(v)$ for all vectors $w$ which is in a sufficiently small neighborhood of a given vector $v \in T^1M$. The topology of $T^1M$ can be given by the Sasaki metric. When restricted to a fiber $T^1_pM$, this topology coincides with the one induced by angle distance (See Definition \ref{angledefinition}). As in \cite{BBE}, we define set of regular vectors $\Re:=\{v: \text{rank}(w)=\ \text{rank}(v)\ \text{for all}\ w\ \text{sufficiently close to}\ v\}$, and $\Re_m:=\{v \in \Re: \text{rank}=m\}$. Obviously, $\Re$ is open and dense in $T^1M$, and $\Re_m$ is open.

A computation from Jacobi equation using $K \leq 0$ gives the convexity of $\|J(t)\|_T=\sqrt{g_{T(t)}(J(t), J(t))}$ in $t$. So we have three special classes of Jacobi fields along geodesic $\gamma_v$:

\begin{enumerate}
\item $J^s(v)$: the space of stable Jacobi fields, with $\|J(t)\|_T$ nonincreasing on $\mathbb{R}$, hence $\|J(t)\|_T \leq c$ for some constant $c$ and $\forall t\geq 0$.
\item $J^u(v)$: the space of unstable Jacobi fields, with $\|J(t)\|_T$ nondecreasing on $\mathbb{R}$, hence $\|J(t)\|_T \leq c$ for some constant $c$ and $\forall t\leq 0$.
\item $J^p(v)$: the space of parallel Jacobi fields, with $\|J(t)\|_T = c$ for some constant $c$ and $\forall t\in \mathbb{R}$.
\end{enumerate}

There is a fundamental difference from the Riemannian case. When we deal with Jacobi fields, even though the reference vector is irrelevant here, the fundamental tensor $g_T$ must be evaluated at $T$. For example, (part of) Rauch comparison theorem for Finsler manifolds is formulated as:

\begin{proposition}
Let $J(t)$ be a Jacobi field along a geodesic $\gamma$ (with velocity vector $T$) with $J(0)=0$ and $g_T(T, J)=0$. If $-a^2 \leq K \leq 0$, the for any $0 < t_1 \leq t_2$,
$$\frac{t_2}{t_1}\leq \frac{\|J(t_2)\|_{T(t_2)}}{\|J(t_1)\|_{T(t_1)}}\leq \frac{s_a(t_2)}{s_a(t_1)},$$
where $\|\cdot\|_T:=\sqrt{g_T(\cdot, \cdot)}$, and $s_a(t)=\frac{1}{a}\sinh (at)$.
\end{proposition}

The length of $J(t)$ is evaluated with reference vector $T$. So if we want to estimate the distance function $d$ from Rauch theorem, we need to deal with the reference vector issue. Nevertheless, we have a coarse estimation of distance due to the following observation:

\begin{proposition}\label{uniform}
There exists a uniform constant $C_0$ for $\widetilde{M}$, such that for $\forall p\in \widetilde{M}$, $\forall v, w \in T^1_p\widetilde{M}$, we have
$$\frac{1}{C_0}\|\cdot\|_v \leq \|\cdot\|_w \leq C_0 \|\cdot\|_v. $$
\end{proposition}

\begin{proof}
Fix a point $p\in \widetilde{M}$, since $T^1_p \widetilde{M}$ is a compact space, there exists $C_0(p)$, such that
$$\frac{1}{C_0(p)}\|\cdot\|_v \leq \|\cdot\|_w \leq C_0(p) \|\cdot\|_v$$
for $\forall v, w \in T^1_p\widetilde{M}$. By Corollary \ref{isometries}, a parallel translation induces a Riemannian isometry $(T_p\widetilde{M}\setminus 0, \hat{g}_p) \rightarrow (T_q\widetilde{M}\setminus 0, \hat{g}_q)$. So in fact $C_0$ can be chosen to be independent of $p$.
\end{proof}

\begin{remark}
In \cite{Eg}, a Finsler manifold with the property in Proposition \ref{uniform} is called a uniform Finsler manifold. For example, compact Finsler spaces, their cover spaces and Berwald spaces are all uniform. The advantage of a uniform Finsler manifold is that we can have a coarse estimation of the distance function by overcoming the reference vector issue partially. In particular in strictly negative curvature case, it works very well and many properties in Riemannian geometry can be extended to uniform Finsler case, see \cite{Eg}.
\end{remark}

\subsection{Construction of flats}

Suppose $\text{rank}(\widetilde{M})\geq 2$ in the rest of this paper. We integrate parallel Jacobi fields to obtain the flats. Recall that Proposition \ref{Foulon} builds an isomorphism between $TT^1\widetilde{M}$ and the space of Jacobi fields. Denote by $\wp(v)$ the preimage of $J^p(v)$ under this isomorphism. For each $m \geq \text{rank}(\widetilde{M})$, distribution $\wp$ has dimension locally constant $m$ on $\Re_m$. We shall prove that $\wp$ is integrable on $\Re_m$.

\begin{lemma}(Lemma 2.1 in \cite{BBE})
The distribution $\wp$ is smooth on $\Re_m$.
\end{lemma}
\begin{proof}
The proof of Lemma 2.1 in \cite{BBE} can be adapted here. We just need to replace the Euclidean inner product $\langle \cdot , \cdot  \rangle$ in the symmetric bilinear form by $g_T(\cdot, \cdot)$. For each vector $v\in \Re_m$ and any $S>0$, consider the symmetric bilinear form
\[Q_S^v(X, Y)=\int_{-S}^Sg_T(R(X, T)T, Y)dt\]
where $T(t)=\gamma_v'(t)$. For a small neighborhood $U$ of $v\in \Re_m$, there exists a large number $S>0$ such that the nullspace of the form $Q_S^w$ is exactly $J^p(w)$ and has constant dimension $m$ for any $w\in U$. So $\wp$ depends smoothly on $w$.
\end{proof}

Given $v\in T^1\widetilde{M}$, let
$$P(v)=\{w\in T^1\widetilde{M}: w\ \text{is parallel to}\ v\}$$
and $F(v)=\pi(P(v))$ where $\pi: T^1\widetilde{M}\to \widetilde{M}$ is the natural projection. By Flat Strip Lemma \ref{flatstripBerwald}, $F(v)$ is a union of flat strips.

\begin{lemma}
$\wp$ is integrable on $\Re_m$, $m \geq \text{rank}(\widetilde{M})$. The maximal arc-connected integral manifold through $v\in \Re_m$ is an open subset of $P(v)$.
\end{lemma}
\begin{proof}
See Lemma 2.2 in \cite{BBE} where Frobenius theorem is used to show the integrability. For Berwald space, we need a coarse estimation to show that the curves tangent to $\wp$ are exactly those curves consisting of parallel vectors.
\begin{enumerate}
\item Let $\sigma: (-\epsilon, \epsilon)\rightarrow \Re_m$ be an integral curve of $\wp$, then $\sigma(s)$ is parallel to $\sigma(0)=v$ for all $s$.

    In fact, consider the geodesic variation $\gamma_s(t)=(\pi \circ g^t)(\sigma(s))$ of $\gamma_0=\gamma_v$. So the variation vector fields
$$Y_s(t)=(d\pi \circ dg^t)(\frac{d}{ds}\sigma(s))$$
are parallel Jacobi fields along $\gamma_s$ by definition of $\wp$. Then for any $t\in \mathbb{R}$,
\begin{equation*}
\begin{aligned}
&d(\gamma_0(t), \gamma_s(t))\leq \int_0^s F(Y_u(t))du\\
\leq &C_0 \int_0^s \|Y_u(t)\|_{T} du
=C_0\int_0^s\|Y_u(0)\|_{T}du,
\end{aligned}
\end{equation*}
where $C_0$ is the uniform constant in Proposition \ref{uniform}, and the last equality is because $Y_u(t)$ is parallel Jacobi fields for any $u$. Hence $\gamma_0$ and $\gamma_s$ are parallel geodesics. So $\sigma(s)$ is parallel to $\sigma(0)=v$ for all $s$.

\item Conversely, if $\sigma(s)\in T^1\widetilde{M}$ is parallel to $\sigma(0)=v$ for all $s$, then $\sigma(s)$ is an integral curve of $\wp$.

    In fact, consider the geodesic variation $\gamma_s(t)=(\pi \circ g^t)(\sigma(s))$ of $\gamma_0=\gamma_v$. $\gamma_s(t)$ are parallel geodesics. Fix arbitrary $s_0$, then $d(\gamma_{s_0}(t), \gamma_s(t))= c(s)$ for $\forall t$ and some constant $c(s)$ dependent on $s$. By Busemann-Mayer theorem:
    $$F(J_{s_0}(t))=\lim_{h\rightarrow 0^+}\frac{d(\gamma_{s_0+h}(t), \gamma_{s_0}(t))}{h}=\lim_{h\rightarrow 0^+}\frac{c(s_0+h)}{s_0} .$$
    Then we have $F(J_{s_0}(t))\leq C$ for any $t$ and some constant $C$ dependent on function $c(s)$. Hence $\|J_{s_0}(t)\|_T\leq C_0 C$ for any $t\in \mathbb{R}$. This implies that $J_{s_0}(t)$ is a parallel Jacobi field along geodesic $\gamma_{s_0}$ for arbitrary fixed $s_0$. Hence $\sigma(s)$ is an integral curve of $\wp$.
\end{enumerate}
The rest argument is the same as in Lemma 2.2 in \cite{BBE}.
\end{proof}

Moreover, given $w\in P(v)\cap \Re_m$, it follows that $P(v)$ and $F(v)$ are smooth $m$-dimensional manifolds near $w$ and $\pi(w)$ respectively (see Lemma 2.2 in \cite{BBE}). Globally, $F(v)$ is a union of flat strips, and it is closed and convex. The following proposition says if $M$ has finite volume, these flat strips join very well to form an $m$-flat.

\begin{proposition}\label{flats}
Suppose that $\widetilde{M}$ admits a quotient manifold $M$ of finite Finsler volume. Then for every $v\in \Re_m$ the set of $F(v)$ is an $m$-flat.
\end{proposition}

The following two lemmas are needed to prove Proposition \ref{flats}. The first one is an analogue of Lemma 2.4 in \cite{BBE}. The notion of orthogonality is with respect to $g_T(\cdot, \cdot)$ instead of the Euclidean inner product $\langle \cdot , \cdot  \rangle$ now. Keeping this in mind, the proof of the following lemma is standard. Similar properties of Jacobi fields can be found in Section 5.4 in \cite{BCS}.
\begin{lemma}\label{orthogonality}
Let $N^*$ be a totally geodesic Berwald submanifold of a Berwald space $N$. Let $\gamma$ be a geodesic of $N$ that lies in $N^*$, and let $J$ be a Jacobi field in $N$ along $\gamma$. Let $J_1$ and $J_2$ denote the components of Y tangent and orthogonal to N * respectively. Then $J_1$ and $J_2$ are Jacobi fields in $N$ along $\gamma$.
\end{lemma}

Let $M=\widetilde{M}/\Gamma$ have finite volume, so does $T^1M$, hence every vector in $T^1M$ is nonwandering relative to the geodesic flow. Lifted to $T^1\widetilde{M}$, for any $w \in T^1\widetilde{M}$, there exist $\{t_n\}\rightarrow +\infty$, $w_n\rightarrow w$ and $\{\phi_n\}\subseteq \Gamma$ the deck group of $M$, such that $(d\phi_n\circ g^{t_n})(w_n)\rightarrow w$ as $n\rightarrow +\infty$. The following lemma is also needed to prove Proposition \ref{flats}. A key ingredient in the proof is Proposition \ref{asymptotic} for Berwald case, and the rest is straightforward.
\begin{lemma}\label{nonwondering}(Lemma 2.5 in \cite{BBE})
Let $M=\widetilde{M}/\Gamma$ have finite volume. Let $v, w$ be asymptotic vectors in $T^1\widetilde{M}$. Then there exist $\{s_n\}\rightarrow +\infty$, $v_n\rightarrow v$ and $\{\phi_n\}\subseteq \Gamma$, such that $(d\phi_n\circ g^{s_n})(v_n)\rightarrow w$ as $n\rightarrow +\infty$.
\end{lemma}



\begin{proof} [Proof of Proposition \ref{flats}]:
Lemma \ref{nonwondering} is used to prove that for arbitrary $q\in F(v)$, $F(v)$ is an $m$-dimensional manifold near $q$. Since as a priori $F(v)$ is already closed and convex, it follows that $F(v)$ is a complete, totally geodesic $m$-dimensional submanifold without boundary. See proof of Proposition 2.3 in \cite{BBE}.
\end{proof}

Next we prove that all the regular vectors have rank all equivalent to $\text{rank}(\widetilde{M})=k$. Here we also borrow a dynamical tool. We say that $v\in T^1\widetilde{M}$ is recurrent if there exists $\{\phi_n\}\subseteq \Gamma$ and $\{t_n\}\rightarrow \infty$ such that $d\phi_n \circ g^{t_n}(v)\rightarrow v$ as $n\rightarrow +\infty$. When $\widetilde{M}$ admits a quotient manifold of finite volume, all recurrent vectors form a dense $G_{\delta}$ set in $T^1\widetilde{M}$.

We need the following four lemmas.

\begin{lemma}(Lemma 2.7 in \cite{BBE})
For $v\in \Re$, denote $Q(v)=S_{\pi(v)}F(v)$. Then:
\begin{enumerate}
\item If $\Im$ is a dense open subset in $T^1\widetilde{M}$, then
$$\Re_0=\{v\in \Re: Q(v)\cap \Im\ \text{is a dense open subset of}\  Q(v)\}$$
is a dense $G_{\delta}$ set in $\Re$.

\item If $\Im$ is a subset of a full measure in $T^1\widetilde{M}$, then
$$\Re_1=\{v\in \Re: Q(v)\cap \Im\  \text{has full measure in}\  Q(v)\}$$
is a subset of full measure in $\Re$.
\end{enumerate}
\end{lemma}
We remark that for Berwald case, the topology and measure on $T^1\widetilde{M}$ above are both equivalent to the ones in Riemannian case. Thus the proof of Lemma 2.7 in \cite{BBE} for Riemannian case can be adapted here.

\begin{lemma}(Lemma 2.8 in \cite{BBE})
Let $v \in \Re$ be recurrent and choose $\{\phi_n\}\subseteq I(\widetilde{M})$ and $\{t_n\}\rightarrow +\infty$ such that $d\phi_n\circ g^{t_n}(v)\rightarrow v$ as $n\rightarrow \infty$. Let $z\in \widetilde{M}(\infty)$ be arbitrary, then all cluster points of the sequence $\{\phi_n(z)\}$ lie in $F(v)(\infty)$.
\end{lemma}

A key ingredient in the proof of the above lemma is Lemma \ref{angleestimate} for Berwald case, which is an important property of our angle notion.

The proofs of the following two lemmas are identical to those in \cite{BBE}.
\begin{lemma} (Lemma 2.9 in \cite{BBE})
Let $x\in \widetilde{M}(\infty)$ be arbitrary and let $A=\overline{I(\widetilde{M})(x)}$, the orbit closure of the isometry group, then $\overline{I(\widetilde{M})(z)}=A$ for every $z\in A$.
\end{lemma}


\begin{lemma}(Lemma 2.10 in \cite{BBE})
If $v,w \in \Re$ are asymptotic, then $\text{rank}(v)=\text{rank}(w)$.
\end{lemma}

With the four lemmas above, one can prove the following main result by adapting the proof of Theorem 2.6 in \cite{BBE}.
\begin{theorem}
Let $k=\text{rank}(\widetilde{M})$. If $\widetilde{M}$ admits a quotient manifold of finite volume. Then:
\begin{enumerate}
\item Every regular has rank $k$.
\item $F(v)$ is a $k$-flat for every $v\in \Re$.
\item Every vector $v\in T^1\widetilde{M}$ is tangent to at least one $k$-flat.
\end{enumerate}

\end{theorem}

\subsection{Strong stable manifolds}

In this subsection we construct strong stable manifolds for uniformly recurrent (see Definition 3.1 in \cite{BBE}) and regular vectors. At first, we prove that any strong stable Jacobi field defined below has an exponentially decreasing length. The main difference is that we must use a coarse estimation on distance functions, based on Proposition \ref{uniform}.

\begin{definition}
A stable Jacobi field $Y$ along $\gamma_v$ (with velocity field T) is called strong stable if $g_T(Y(0),Z(0))=0$ for any $Z\in J^p(v)$, where $g_T=g_{ij}(\gamma_v, T)$. Denote by $J^{ss}(v)$ the space of all strong stable Jacobi fields along $\gamma_v$.
\end{definition}

For $Y\in J^{ss}(v)$, $g_T(Y(t),Z(t))\equiv 0$ for any $Z \in J^p(v)$. The proof is same as in Lemma 3.3 in \cite{BBE}, with inner product replaced by $g_T$ norm. Under this modification, Lemma 3.4 in \cite{BBE} is also true for Berwald case. Here we let $H(v)$ denote the horosphere determined by $v$, and let $W(v)=\{v(q): q\in H(v)\}$.

\begin{lemma}\label{decreasing}
Let $v\in \Re$ be a uniformly recurrent vector. Then there exists a neighborhood $U$ of $v$ in $W(v)$ and constants $C$ and $\lambda=\lambda(v)>0$ such that for every $w\in U$, $Y\in J^{ss}(w)$ and $t>0$
$$\|Y(t)\|_T \leq Ce^{-\lambda t}\|Y(0)\|_T.$$
\end{lemma}

\begin{definition}\label{strongstable}(Lemma 3.4 in \cite{BBE})
For $v \in \Re$, and $w\in W(v)\cap \Re \subseteq T^1\widetilde{M}$ define
$$D_v^s(\pi(w))=\{Y(0)|Y\in J^{ss}(w)\}.$$
\end{definition}

Recall the Busemann function $f_{\gamma}:\widetilde{M}\rightarrow R$ is defined as
$$f_{\gamma}(p)=\lim_{t\rightarrow \infty}(d(p, \gamma(t))-t),$$
which is a convex function. In \cite{Eg}, it is proved that $f_{\gamma}$ is at least $C^1$ and $df_{\gamma}(Y)=g_X(X,Y)$ where $X=-\alpha'(0)$, and $\alpha(t)$ is the unique geodesic from $\pi(Y)$ to $\gamma(\infty)$. Moreover, as in \cite{HI}, we can show that radial $X(p)$ is continuously differentiable in $p$. Hence $f_{\gamma}$ is $C^2$. Now in Definition \ref{strongstable}, if $Y \in J^{ss}(w)$ and $w\in W(v)\cap \Re$, then  $g_w(Y(0),w)=0$, hence $Y(0)$ is tangent to horosphere $H(v)$ at $\pi(w)$. Hence $D_v^s(\pi(w))$ is a distribution on $H(v)$. Since $f_{\gamma}$ and hence $H(v)$ is $C^2$, $D_v^s$ is a $C^1$ distribution on an open set of $\pi(v)$ on $H(v)$. We integrate $D_v^s$ on $H(v)$.

\begin{lemma}(Lemma 3.6 in \cite{BBE})
If $v\in \Re$ is uniformly recurrent in $T^1\widetilde{M}$, then $D_v^s$ is integrable in a neighborhood of $\pi(v)$ in $H(v)$.
\end{lemma}

\begin{proof}
See Lemma 3.6 in \cite{BBE}. We only need to modify a little bit when we prove that if $\alpha(s)$ is a piecewise $C^1$ curve tangent to $D_v^s$ with $\alpha(0)=p$, $\alpha(1)=q$, then $d(\gamma_{v(p)}(t), \gamma_{v(q)}(t))\rightarrow 0$ as $t\rightarrow \infty$. This follows from the coarse estimation of the distance function and Lemma \ref{decreasing}: there is a geodesic variation $\gamma_{v(\alpha(s))}(t)$ with Jacobi fields $J_s(t)$
\begin{equation*}
\begin{aligned}
d(\gamma_{v(p)}(t), \gamma_{v(q)}(t))&\leq \int_0^1 F(J_s(t))ds\leq \int_0^1 C_0\|J_s(t)\|_T ds
\\ &\leq C_0C e^{-\lambda t}\|J_s(0)\|_T \to 0 \ \  \ \text{\ as\ } t\to \infty,
\end{aligned}
\end{equation*}
where $C_0$ is the constant in Proposition \ref{uniform}.
\end{proof}

\begin{proposition}(Proposition 3.7 in \cite{BBE})
For every $v \in \Re$, $D_v^s$ is integrable on $\pi(W(v)\cap \Re)\subseteq H(v)$.
\end{proposition}

\begin{proof}
See Proposition 3.7 in \cite{BBE}. When we construct the approximating vector fields, note that we use the fact that radial vector fields $X(q)$ is $C^1$ in $q$; and orthogonal projection is with respect to the norm $g_{u}$.
\end{proof}

Denote $H^{ss}(v)$ the integral manifold of $D_v^s$ near $\pi(v)$ on $H(v)$. Then $H^{ss}(v)$ is an $(n-k)$-dimensional submanifold which intersects $F(v)$ transversally and orthogonally with respect to $g_v$, where $k=\text{rank\ }(M)$. The strong stable manifold of $v$ is defined as
$$W^{ss}(v)=\{w\in W(v)| \pi(w)\in H^{ss}(v)\}.$$

The following proposition says $W^{ss}(v)$ is indeed a strong stable manifold in dynamics sense.
\begin{proposition}\label{exponential}(Proposition 3.10 in \cite{BBE})
Let $v \in \Re$ be uniformly recurrent. Then
$$W^{ss}(v)=\{w\in W(v)|d(g^tv, g^t w)\rightarrow 0\ \text{as}\ t\rightarrow \infty\},$$
and if $w\in W^{ss}(v)$, then
$$d(g^tv,g^tw)\leq Cd_s(v,w)e^{-\lambda t}$$
for all $t\geq 0$, where $\lambda =\lambda(v) >0$, $C=C(w)$ is positive and continuous on $W^{ss}(v)$, $d_s$ is the induced distance in $W^{ss}(v)$.
Moreover, $W^{ss}(v)$ is a closed connected submanifold of $T^1\widetilde{M}$ which is diffeomorphic to $\mathbb{R}^{n-k}$.
\end{proposition}

\begin{definition}
For any vector $v\in \Re$, the strong unstable manifold of $v$ is defined as $W^{uu}(v)=-W^{ss}(-v)$.
\end{definition}

\subsection{Weyl chambers}

Before defining Weyl chambers, let us discuss the following two lemmas, which are important tools to show all the Weyl chambers are isometric under the angle metric.

\begin{lemma} \label{convexitylemma}(Convexity Lemma, Lemma 1.5 in \cite{BBE})
Consider $v, w\in T^1_p\widetilde{M}$, $v',w' \in T^1_{p'}\widetilde{M}$. Suppose $v$ is asymptotic to $v'$, $w$ is asymptotic to $w'$, and both $\triangle(v,w)$ and $\triangle(v', w')$ are flat triangles. Then
\begin{enumerate}
\item $\angle(v,w)=\angle(v',w')=\theta >0$.
\item For any number $\alpha$ with $0\leq \alpha \leq \theta$, let $v_{\alpha}$, $v'_{\alpha}$ denote vectors at $p$ and $p'$ tangent to $\triangle(v,w)$, $\triangle(v',w')$ such that $\angle(v,v_{\alpha})=\angle(v',v'_{\alpha})=\alpha$. Then $v_{\alpha}$ and $v'_{\alpha}$ are asymptotic.
\end{enumerate}
\end{lemma}

\begin{proof}
Recall the angle notion in Definition \ref{angledefinition}. In flat triangles, by Proposition \ref{angle}, the following relations are also true for Berwald case:
$$d(\gamma_v(t), \gamma_{w}(t))=t\theta\ \ \text{and}\ \ d(\gamma_{v'}(t), \gamma_{w'}(t))=t\theta'.$$
then the lemma follows from the same argument in Lemma 1.5 in \cite{BBE}.
\end{proof}

Denote $\lambda(v)= -[\max_{Y\in J^s(v)}\limsup_{t\rightarrow +\infty}\frac{\ln\|Y(t)\|_T}{t}]$. By Proposition \ref{exponential}, $\lambda(v)>0$ if $v$ is uniformly recurrent. Suppose $-b^2 \leq K \leq 0$.

\begin{lemma}\label{anglelemma}(Angle Lemma, Lemma 4.4 in \cite{BBE})
Let $v\in \Re$ be uniformly recurrent. Choose $\alpha \in[0, \frac{\pi}{4}]$ such that $0< \alpha< \frac{\lambda(v)}{b}$ and let $0< \beta<\lambda(v)-\alpha b$ be chosen arbitrarily. Then for any $w \in T_{\pi(v)}F(v)$ with $\angle(v, w) \leq \alpha$:
\begin{enumerate}
\item For every $q\in H^{ss}(v)$, $w(q)$ is tangent to $F(v(q))$. Moreover,
$$\angle(v,w)\\ =\angle(v(q),w(q)).$$
\item If $q\in H^{ss}(v)$, then
  $$d(g^tw, g^tw(q))\leq Ce^{-\beta t}d(w,w(q)),$$
where $C=C(q)$ is bounded on compact subsets of $H^{ss}(v)$.
\item $w \in \Re$.
\end{enumerate}
\end{lemma}

\begin{proof}
The proof is same as Lemma 4.4 in \cite{BBE}. But the following coarse estimation should be used:
\begin{enumerate}
\item If $v, w \in T^1_p\widetilde{M}$, $\frac{1}{C_1} \angle(v,w)\leq d^{*}(v,w)\leq C_1 \angle(v,w)$.
\item $\frac{1}{C_0}F(J)\leq \|J\|_T \leq C_0F(J)$.
\end{enumerate}
where $C_0$ is the constant in Proposition \ref{uniform}. The notations are same as in Lemma 4.4 in \cite{BBE}, for example $w(q)$ denotes the unique vector asymptotic to $w$ and with foot point $q$; and $d^*$ is induced by the Sasaki metric on $T^1\widetilde{M}$.
\end{proof}

Now we define Weyl chambers for $p$-regular vectors. A vector $v\in T^1\widetilde{M}$ is called $p$-regular if there is a point $q\in F(v)$ such that $v(q)$ is regular. If $v$ is $p$-regular, we set
$$A(v)=\{q\in \widetilde{M}|v(q)\ \text{is}\ p\text{-regular}\}.$$
It can be proved that if $v\in \Re$ is uniformly recurrent, then $A(v)=\widetilde{M}$.

\begin{definition}
For a $p$-regular vector $v$, we can define two types of Weyl chambers:
\begin{enumerate}
\item $\mathcal{C}(v)$ is the set of all $p$-regular vectors $w\in S_{\pi(v)}F(v)$ such that $A(w)=A(v)$ and $F(w(q))=F(v(q))$ for all $q\in A(v)$.
\item $\tilde{\mathcal{C}}(v):=\{w\in S_{\pi(v)}F(v)| w(q)\in S_qF(v(q))\ \text{for all}\ q \in A(v)\}$.
\end{enumerate}

\end{definition}

Clearly $\mathcal{C}(v)\subseteq \tilde{\mathcal{C}}(v)$. Since isometries of $\widetilde{M}$ and parallel translations of $F(v)$ preserve angles, they induce isometries (with respect to angle distance) between two Weyl Chambers. Moreover, by Convexity Lemma 4.3.22, there are isometries  $\mathcal{C}(v)\rightarrow \mathcal{C}(v(q))$ and $\tilde{\mathcal{C}}(v)\rightarrow \tilde{\mathcal{C}}(v(q))$ both defined by $w\rightarrow w(q)$. As a priori, $\mathcal{C}(v)$ is not necessarily open, but we can prove that $\tilde{\mathcal{C}}(v)$ is closed and convex. See Lemma 2.5 in \cite{BBS}. Furthermore, by Angle Lemma \ref{anglelemma} we have:

\begin{lemma}(Lemma 2.7 in \cite{BBS})
If $v$ is uniformly recurrent, then
\begin{enumerate}
\item $\mathcal{C}(v)$ contains an open neighborhood of $v$ in $S_{\pi(v)}F(v)$;
\item $\tilde{\mathcal{C}}(v)$ is the closure of $\mathcal{C}(v)$;
\item $\mathcal{C}(v)$ is convex.
\end{enumerate}
\end{lemma}

When we prove that two Weyl chambers are isometric, we use limit arguments. The following two lemmas are the main tools. See Lemma 2.8, 2.9 in \cite{BBS}. Given a sequence of subsets $X_n$ in a space $X$, denote
$$\overline{\lim}X_n=\{x\in X| x\ \text{is a limit point of a sequence}\ x_n \in X_n\}.$$

\begin{lemma}(Lemma 2.8 in \cite{BBS})
Let a sequence $v_n \in \Re$ converge to $v\in \Re$. Then $\overline{\lim}\tilde{\mathcal{C}}(v_n)$ is contained in $\tilde{\mathcal{C}}(v)$.
\end{lemma}
\begin{lemma}(Lemma 2.9 in \cite{BBS})
Let a sequence $v_n \in \Re$ converge to $v\in \Re$. For $\alpha>0$ let $\mathcal{C}_{\alpha}(v_n)$ denote the $\alpha$-interior (under angle distance) of $\mathcal{C}(v_n)$ in $S_{\pi(v_n)}F(v_n)$. Assume that for all $n$, $v_n \in \mathcal{C}_{\alpha_0}(v_n)$ for some $\alpha_0 >0$. Then for all positive $\alpha \leq \alpha_0$
$$\overline{\lim} ~\mathcal{C}_{\alpha}(v_n) \subset \mathcal{C}(v).$$
\end{lemma}
Let $\mathcal{B}$ denote the vectors which are uniformly recurrent in both positive and negative direction. The following result is important while the proof is identical to that of Lemma 2.10 in \cite{BBS}:

\begin{lemma}(Rigidity Lemma, Lemma 2.10 in \cite{BBS})\label{rigiditylemma}
Let $v\in \mathcal{B}$. Then there is an open set $U=U(v)\subset \Re$ containing $v$ such that for all $w\in U$ we have $w\in \text{Int}\ \tilde{\mathcal{C}}(w)$ and $\text{Int}\ \tilde{\mathcal{C}}(w)$ is isometric to $\text{Int}\ \tilde{\mathcal{C}}(v)$.
\end{lemma}

So locally around a vector in $\mathcal{B}$, all Weyl chambers are isometric.

We can construct first integrals on an $g^t$-invariant dense open subset of $T^1M$, which is an analogue of Theorem 3.7 in \cite{BBS}.
The proof is identical.
\begin{theorem}(First integrals, Theorem 3.7 in \cite{BBS})
Let $M$ be a quotient manifold of $\widetilde{M}$ of finite volume, then there is a $g^t$-invariant open and dense subset $R$ of $T^1M$ and $k-1$ independent $C^1$ first integrals
$$\Phi_i: R\rightarrow \mathbb{R},\ \ \ 1\leq i \leq k-1,$$
such that each $v\in R$ has a neighborhood $R(v)$ in $R$ with the following property:
If $v^{*}\in SF_{R(v)}(v')$, then $\Phi_i(v^{*})=\Phi_i(v')$ for all $i$ if and only if $v^{*}$ is parallel to $v'$ in $F(v')$, where
$SF_{R(v)}$ is the foliation induced by $SF$ in $R(v)$.
\end{theorem}

Denote by $I(v)$ the level set of the first integrals $\Phi_1,\ldots, \Phi_{k-1}$ containing $v$. For notations, refer to Section 4 in \cite{BBS}. In a neighborhood of a vector in $\Re^*$, a dense open subset of $\Re$, the following Anosov type closing lemma holds:

\begin{lemma}(Closing lemma, Lemma 4.5 in \cite{BBS})
Denote $I^{*}(v_0):=I(v_0)\cap \Re^*$. For every compact $K\subset I^{*}(v_0)$ and any $\epsilon \in(0,1)$ there exists $T_0=T_0(K)>0$ and $\delta=\delta(K,\epsilon)>0$ with the following property: if
$$d(d\phi(g^Tv),v)<\delta$$
for some $v\in K$, $T\geq T_0$ and for some deck transformation $\phi$, then there are $v'\in R\subset \Re$ and $T' \in [T-\epsilon, T+\epsilon]$ such that
\begin{enumerate}
\item $d(g^tv,g^tv')<\epsilon$ for all $t\in [0,T]$,
\item $d\phi\circ g^{T'}v'=v'$,
\item $\phi$ is a pure translation of the flat $F(v')$.
\end{enumerate}
\end{lemma}


\begin{corollary}
If $M$ has finite volume and the flag curvature is bounded below then vectors tangent to regular closed geodesics are dense in $T^1M$.
\end{corollary}

\subsection{Tits building at infinity}
In \cite{BS2}, Burns and Spatzier constructed Tits buildings for $\widetilde{M}(\infty)$ using Weyl chambers (or called Weyl simplices in terms of Tits building) in the above subsection. They call $v\in T^1\widetilde{M}$ $l$-regular if $v$ is asymptotic to a regular vector. The notion of Weyl simplices can be extended to $l$ regular vectors with little modification. The first important work of Burns and Spatzier is to show that all Weyl simplices for $l$-regular vectors are isometric to each other. Hence Weyl simplices can be constructed at $\widetilde{M}(\infty)$. See \cite{BS2}. The results are also true for our Berwald case.

\begin{theorem}\label{isometric}
All Weyl simplices for $l$-regular vectors are isometric. If $v\in T^1\widetilde{M}$ is $l$-regular, then $\tilde{\mathcal{C}}(v)$ is a $(k-1)$-dimensional convex subset of $T^1_{\pi(v)}\widetilde{M}$.
\end{theorem}

\begin{proof}
See Section 2 in \cite{BS2}. The main tools are the Rigidity Lemma \ref{rigiditylemma} and one technical lemma which is stated below.
\end{proof}

\begin{lemma}(Lemma 1.1 in \cite{BS2})
Let $\gamma_v$ be a periodic regular geodesic and $\phi$ be an axial isometry for $\gamma_v$. Then for any $x\in \widetilde{M}(\infty)$, $n\geq 0$,
$$\angle_{\gamma_v(0)}(\phi^nx, \gamma_v(\infty))\leq \angle_{\gamma_v(0)}(x, \gamma_v(\infty)).$$
If $n$ is large enough, quality holds if and only if $x\in F(v)(\infty)$. Any limit point of $\{\phi^nx: n\geq 0\}$ lies in $F(v)(\infty)$.
\end{lemma}

\begin{proof}
See Lemma 1.1 in \cite{BS2}. The argument works for Berwald space because of Lemma \ref{angleestimate}.
\end{proof}

\begin{remark}
In fact, we also have a parallel inequality:
$$\angle_{\gamma_v(0)}(\phi^{-n}x, \gamma_v(\infty))\geq \angle_{\gamma_v(0)}(x, \gamma_v(\infty)),$$
and the equality holds for large $n$ if and only if $x\in F(v)(\infty)$. The idea of this fact is used in the proof of Theorem \ref{isometric}. See \cite{BS2}.
\end{remark}

Now we can define Weyl simplices at $\widetilde{M}(\infty)$.

\begin{definition}
Let $x\in \widetilde{M}(\infty)$ with a $l$-regular geodesic representative $\gamma_v$. The Weyl simplex of $x$ is defined as:
$$C(x):= \{\gamma_w(\infty): w\in \tilde{\mathcal{C}}(v)\}.$$
\end{definition}

By Convexity Lemma \ref{convexitylemma} and Theorem \ref{isometric}, this definition doesn't depend on the choice of $l$-regular geodesic representative. If we introduce Tits distance on $\widetilde{M}(\infty)$, $C(x)$ is a $(k-1)$-dimensional convex subset.

\begin{definition}
A spherical Tits building is a simplicial complex $\Delta$ together with a family $\{\Sigma\}$ of finite subcomplexes called apartments satisfying the following axioms:

\begin{enumerate}
\item $\Delta$ is thick, i.e, every codimensional $1$ simplex in a top dimensional simplex is contained in at least 3 top dimensional simplices;
\item every apartment is a Coxeter complex;
\item any two elements of $\Delta$ belong to an apartment;
\item if $\Sigma$ and $\Sigma'$ are two apartments containing both $A$ and $A'\in \Delta$, then there is an isomorphism of $\Sigma$ onto $\Sigma'$ which leaves $A, A'$ and all their faces invariant.

\end{enumerate}
\end{definition}
Now let $\Delta$ be the set consisting of the Weyl simplices in $\widetilde{M}(\infty)$ and all their intersections. The subcomplex $\Sigma$ that is isomorphic to a complex $\Sigma_F$ for some regular $k$-flat $F$ is called an apartment if the union of its Weyl simplices is homeomorphic to a $(k-1)$-sphere. Let $\mathcal{A}$ be the collection of all apartments in $\Delta$. In \cite{BS2}, axioms 1-4 are verified for $(\Delta, \mathcal{A})$. Similarly it is true for Berwald case, hence

\begin{theorem}
$(\Delta, \mathcal{A})$ is a spherical Tits building.
\end{theorem}

We need to use the main theorem in \cite{BS1}.

\begin{theorem}\label{titsbuilding}(cf. \cite{BS1})
Let $\Delta$ be an infinite, irreducible, locally connected, compact, metric, topologically Moufang building of rank at least 2. Then $\Delta$ is the building of parabolic subgroups of a real simple Lie group $G$.
\end{theorem}

If $M$ is a complete Berwald space with flag curvature $-b^2\leq K \leq 0$, of finite volume, and rank $k\geq 2$ whose universal cover $\widetilde{M}$ is irreducible, we can prove that the spherical Tits building $(\Delta, \mathcal{A})$ at $\widetilde{M}(\infty)$ is infinite, irreducible, locally connected, compact, metric, topologically Moufang building of rank $k$. See Section 3 and 4 in \cite{BS2}.

\subsection{Classification}
The last step is to adapt the arguments of Gromov's Rigidity Theorem (Chapter 4 in \cite{BGS}) to prove Main Theorem \ref{higherrank}. We also use the results on locally symmetric Finsler spaces from \cite{DH}.

Now let $G$ be the topological automorphism group of $\Delta$ and $G^0$ be the connected component of identity in $G$. So $G^0$ is a simple noncompact real Lie group without center. Let $\Delta(G^0)$ be the topological building of parabolic subgroups of $G^0$. By Theorem \ref{titsbuilding}, $\Delta(G^0)$ is isomorphic to $\Delta$. Moreover, if let $X=G^0/K$ be the symmetric space attached to $G^0$ and $\Delta(X)$ constructed at $X(\infty)$, then $\Delta(X)$ is isomorphic to $\Delta(G^0)$, and hence $\Delta(X)$ is isomorphic to $\Delta$ too.

There is a well defined so called Tits metric on $\widetilde{M}(\infty)$. We want to carry over this metric from $\widetilde{M}(\infty)$ to $X(\infty)$ via the isomorphism $\phi: \Delta\to \Delta(X)$ described above. Take any $x\in \widetilde{M}(\infty)$, the geometric structure of Weyl simplex $C(x)$ can be identified with $\tilde{\mathcal{C}}(v)$, a convex subset of $(k-1)$-sphere, for some $v\in T^1\widetilde{M}$ such that $\gamma_v(\infty)=x$. Since $\phi(C(x))$ is a Weyl simplex in $\Delta(X)$, we identify it with a Weyl simplex $\tilde{\mathcal{C}}(v^{*})$ for some $v^*\in T^1X$. Hence there exists an isomorphism between two simplices stilled denoted as $\phi: \tilde{\mathcal{C}}(v)\to \tilde{\mathcal{C}}(v^{*})$, i.e, if $\{v_1, v_2,..., v_k\}$ are vertices of $\tilde{\mathcal{C}}(v)$, then $\{\phi(v_1), \phi(v_2), ..., \phi(v_k)\}$ are vertices of $\tilde{\mathcal{C}}(v^{*})$. $\phi$ can be linearly extended to the map $\phi: \tilde{\mathcal{C}}(v)\rightarrow \tilde{\mathcal{C}}(v^{*})$. Recall that $F$ is a Minkowski norm in $T_{\pi(v)}\widetilde{M}$. Let $F^{*}= F\circ \phi^{-1}$ and extend it to the cone spanned by $\{\phi(v_1), \phi(v_2), ..., \phi(v_k)\}$ via $F^{*}(\lambda w)=\lambda F(w)$ for $\forall \lambda >0$ and $\forall w\in \tilde{\mathcal{C}}(v^{*})$. By Theorem 4.6 in \cite{DH} (and its proof), $F^{*}$ can be extended to a Minkowski norm $F^{*}$
on $T_{\pi(v^{*})}X$, and then extended to $TX$ by left translations. Hence we obtain a locally symmetric Finsler space $(X, F^{*})$.

Next we adapt the argument in \cite{BS2} to construct an isometry $\Phi$ between $(\widetilde{M}, F)$ and $(X, F^{*})$. Notations are same as in \cite{BS2}.

First we define $\Phi:\widetilde{M}\rightarrow X$. Take $p\in \widetilde{M}$. The geodesic symmetry $\sigma_p$ defines a continuous automorphism of $\Delta$. By 16.2 in \cite{M} , $\sigma_p$ determines an analytic involuntary isomorphism $\Theta_p$ of $G^0$. $\Theta_p$ induces an isometry $\theta_p: (X, F^{*})\rightarrow (X, F^{*})$. $\theta_p$ has order $2$, and it has a unique fixed point $p^{*}$. Set $\Phi(p)=p^{*}$. The proofs of the following lemmas are identical to those in \cite{BS2}.

\begin{lemma}(Lemma 5.2 in \cite{BS2})
\begin{enumerate}
\item $\Phi: \widetilde{M}\rightarrow X$ is continuous.
\item If $F\subset \widetilde{M}$ is an $l$-regular $k$-flat, then $\Phi(F)\subset F^*$, where $F^*$ is the unique $k$-flat in $X$ such that $\Sigma_{F^*}=\Sigma_F$ under the isomorphism.
\end{enumerate}
\end{lemma}

Let $\gamma$ be a maximally singular geodesic, that is, $\gamma(\infty)$ is a vertex of $\Delta$. Let $C_1$ and $C_2$ be two opposite chambers in Star $\gamma(\infty)$. Then $C_1 \cap C_2=\{\gamma(\infty)\}$. Let $F_i$ be the $l$-regular $k$-flat through $\gamma(0)$ and $C_i$. Then $F_1\cap F_2=\gamma$. By Lemma 3.30, $\Phi(\gamma)\subset F_1^*\cap F_2^*$. Since $F_1^*(\infty)\cap F_2^*(-\infty)=\{\gamma(\infty),\gamma(-\infty)\}$, $F_1^*\cap F_2^*$ is a maximally singular geodesic in $X$, denoted as $\gamma^*$. Since Tits distance on $\widetilde{M}(\infty)$ and $X(\infty)$ are isometric, we have:

\begin{lemma}
\begin{enumerate}
\item If $\gamma_1$ and $\gamma_2$ are two parallel maximally singular geodesics, then $\gamma_1^*$ and $\gamma_2^*$ are parallel.
\item If $\gamma_1$ and $\gamma_2$ are two maximally singular geodesics, the the families of geodesics parallel to $\gamma_1^*$ and $\gamma_2^*$ make same angle as do those parallel to  $\gamma_1$ and $\gamma_2$.
\end{enumerate}
\end{lemma}

\begin{lemma}(Lemma 5.3 in \cite{BS2})
Let $\gamma$ be a maximally singular geodesic then $\Phi|\gamma: \gamma\rightarrow \gamma^*$ is affine.
\end{lemma}

\begin{lemma}
Let $F$ be a $l$-regular $k$-flat, then $\Phi|F: F\rightarrow F^*$ is affine.
\end{lemma}

Since every geodesic of $\widetilde{M}$ lies in an $l$-regular flat, for each geodesic $\gamma$, there is a constant $\lambda(\gamma)$ such that
$$d_X[\Phi(q_1), \Phi(q_2)]= \lambda(\gamma)d(q_1, q_2)$$
for any points $q_1$ and $q_2$ on $\gamma$.

\begin{lemma}(Lemma 5.4 in \cite{BS2})
Let $p$ be a point of an $l$-regular $k$-flat $F$. Then $\lambda(\gamma)$ is the same for all geodesics $\gamma$ with $\gamma'(0)\in S_pF$.
\end{lemma}

Hence $d_X[\Phi(q_1), \Phi(q_2)]= \lambda(F)d(q_1, q_2)$ for any point $q_1, q_2\in F$.

\begin{lemma}
$\lambda(F)$ is independent of $F$.
\end{lemma}

Hence $d_X(\Phi(q_1), \Phi(q_2))=\lambda d(q_1, q_2)$ for any $q_1, q_2 \in \widetilde{M}$. And obviously $\lambda \neq 0$. Hence we have an isometry between $(\widetilde{M}, F)$ and $(X,F^*)$ up to a renormalization. Hence we have proved Main Theorem \ref{higherrank}.

\ \
\\[-2mm]
\textbf{Acknowledgement.} This paper is a part of my Ph.D. thesis and I would like to devote it to Anatole Katok, whose encouragement and enthusiasm are always precious to me. I also thank Federico Rodriguez Hertz for helpful conversations, and Vladimir S. Matveev, Zolt\'{a}n I. Szab\'{o} and the referees for many valuable suggestions.  This work is partially supported by NSFC Nos. 11701559 and 11571387.


\begin{thebibliography}{10}

\bibitem{Ba}
W. Ballmann, \emph{Nonpositively curved manifolds of higher rank}, Annals of Mathematics 122.3 (1985): 597-609.

\bibitem{BBE}

W.~Ballmann, M.~Brin and P.~Eberlein, \emph{Structure of manifolds of nonpositive curvature. I}, Ann. of Math 122.2 (1985): 171-203.

\bibitem{BBS}

W. Ballmann, M. Brin and R. Spatzier, \emph{Structure of manifolds of nonpositive curvature. II}, Ann. of Math.(2) 122.2 (1985): 205-235.


\bibitem{BGS}
W.~Ballmann, M.~Gromov and V. Schroeder, \emph{Manifolds of Nonpositive Curvature}, Birkh\"{a}user, Progress in Mathematics, 61 (1985), vi+263 pp.

\bibitem{BCS}
D.~Bao, S.-S. Chern and Z. Shen, \emph{An introduction to Riemann-Finsler geometry}, Graduate Texts in Mathematics, Vol. 200. Springer, 2000, xx+431 pp.


\bibitem{BS1}
K. Burns and R. Spatzier, \emph{On topological Tits buildings and their classification}, Publications Math\'{e}matiques de l'Institut des Hautes \'{E}tudes Scientifiques 65.1 (1987): 5-34.

\bibitem{BS2}
K. Burns and R. Spatzier, \emph{Manifolds of nonpositive curvature and their buildings}, Publications Math\'{e}matiques de l'Institut des Hautes \'{E}tudes Scientifiques 65.1 (1987): 35-59.

\bibitem{DH}
S.~Deng and Z.~Hou, \emph{Positive definite Minkowski Lie algebras and bi-invariant Finsler metrics on Lie groups}, Geometriae Dedicata 136.1 (2008): 191-201.

\bibitem{EO}
P.~Eberlein and B.~O'Neill, \emph{Visibility manifolds}, Pacific Journal of Mathematics 46, no. 1 (1973): 45-109.

\bibitem{Eg}
D. Egloff, \emph{Uniform Finsler Hadamard manifolds}, Annales de l'IHP Physique th\'{e}orique. Vol. 66. No. 3 (1997): 323-357

\bibitem{Fo}
P.~Foulon, \emph{Curvature and global rigidity in Finsler manifolds}, Houston J. Math 28.2 (2002): 263-292.

\bibitem{HI}
E. Heintze and H.-C. Im Hof, \emph{Geometry of horospheres}, Journal of Differential Geometry 12.4 (1977): 481-491.

\bibitem{J}
J.~Jost, \emph{Nonpositive curvature: geometric and analytic aspects}, Lectures in Mathematics ETH Z\"{u}rich, Birkh\"{a}user Verlag, Basel, 1997, viii+108 pp.

\bibitem{KK}
A. Krist¨¢ly and L. Kozma, \emph{Metric characterization of Berwald spaces of non-positive flag curvature}, Journal of Geometry and Physics 56.8 (2006): 1257-1270.


\bibitem{M}
G. D. Mostow, \emph{Strong rigidity of locally symmetric spaces}, Annals of Mathematics Studies, No. 78. Princeton University Press, Princeton, N.J., 1973, v+195 pp.

\bibitem{Sh}
Z. Shen, \text{Differential geometry of spray and Finsler spaces}, Kluwer Academic Publishers, Dordrecht, 2001, xiii+258 pp.



\end{thebibliography}
\end{document}